\documentclass[a4paper,UKenglish]{article}

\usepackage[utf8]{inputenc}
\IfFileExists{lmodern.sty}{\usepackage{lmodern}}{}
\usepackage[T1]{fontenc}
\usepackage{babel}
\usepackage{amssymb, amsmath, amsthm}

\theoremstyle{plain}
\newtheorem{theorem}{Theorem}
\newtheorem{lemma}[theorem]{Lemma}
\theoremstyle{definition}
\newtheorem{definition}[theorem]{Definition}
\theoremstyle{remark}
\newtheorem*{remark}{Remark}

\usepackage{microtype, booktabs, fullpage}
\newcommand{\eps}{\varepsilon}
\newcommand{\E}{\mathbb{E}}
\newcommand{\bigO}{\mathcal{O}}
\newcommand{\R}{\mathbb{R}}

\title{On Using Toeplitz and Circulant Matrices for
  Johnson-Lindenstrauss Transforms}

\author{Casper Benjamin Freksen\thanks{This research is supported by a Villum
    Young Investigator Grant, an AUFF Starting Grant and MADALGO,
	Center for Massive Data Algorithmics,
	a Center of the Danish National Research Foundation, grant DNRF84.
	Department of Computer Science, Aarhus University.
	\{\texttt{cfreksen,larsen}\}\texttt{@cs.au.dk}.}
 \and
 	Kasper Green Larsen\footnotemark[1]}

\usepackage{hyperref}

\begin{document}

\date{}
\maketitle

\begin{abstract}
  The Johnson-Lindenstrauss lemma is one of the corner stone results
  in dimensionality reduction. It says that given $N$, for any set of
  $N$ vectors $X \subset \R^n$, there exists a mapping
  $f : X \to \R^m$ such that $f(X)$ preserves all pairwise distances
  between vectors in $X$ to within $(1 \pm \eps)$ if
  $m = \bigO(\eps^{-2} \lg N)$. Much effort has gone into developing
  fast embedding algorithms, with the Fast Johnson-Lindenstrauss
  transform of Ailon and Chazelle being one of the most well-known
  techniques. The current fastest algorithm that yields the optimal
  $m = \bigO(\eps^{-2}\lg N)$ dimensions has an embedding time of
  $\bigO(n \lg n + \eps^{-2} \lg^3 N)$. An exciting approach towards
  improving this, due to Hinrichs and Vybíral, is to use a random
  $m \times n$ Toeplitz matrix for the embedding. Using Fast Fourier
  Transform, the embedding of a vector can then be computed in
  $\bigO(n \lg m)$ time. The big question is of course whether
  $m = \bigO(\eps^{-2} \lg N)$ dimensions suffice for this
  technique. If so, this would end a decades long quest to obtain
  faster and faster Johnson-Lindenstrauss transforms. The current best
  analysis of the embedding of Hinrichs and Vybíral shows that
  $m = \bigO(\eps^{-2}\lg^2 N)$ dimensions suffices. The main result
  of this paper, is a proof that this analysis unfortunately cannot be
  tightened any further, i.e., there exists a set of $N$ vectors
  requiring $m = \Omega(\eps^{-2} \lg^2 N)$ for the Toeplitz approach
  to work.
\end{abstract}

\section{Introduction}
\label{sec:introduction}

The performance of many geometric algorithms depends heavily on the
dimension of the input data. A widely used technique to combat this
``curse of dimensionality'', is to preprocess the input via
\emph{dimensionality reduction} while approximately preserving
important geometric properties. Running the algorithm on the lower
dimensional data then uses less resources (time, space, etc.) and an
approximate result for the high dimensional data can be derived from
the low dimensional result.

Dimensionality reduction approximately preserving pairwise
Euclidean distances has found uses in a wide variety of
applications, including: Nearest-neighbour search
\cite{Ailon:2009:tFJLTaANN,Har-Peled:2012:ANNtRtCoD}, clustering
\cite{Boutsidis:2015:RDRfkMC,Cohen:2015:DRfkMCaLRA}, linear
programming \cite{Vu:2015:UtJLLiLaIP}, streaming algorithms
\cite{Muthukrishnan:2005:DSAaA}, compressed sensing
\cite{Candes:2006:RUPESRfHIFI,Donoho:2006:CS}, numerical linear
algebra \cite{Woodruff:2014:SaaTfNLA}, graph sparsification
\cite{Spielman:2011:GSbER}, and differential privacy
\cite{Blocki:2012:tJLTIPDP}. See more applications in
\cite{Vempala:2004:tRPM,Indyk:2001:AAoLDGE}. The most fundamental
result in this regime is the Johnson-Lindenstrauss
(JL) lemma \cite{Johnson:1984:EoLMiaHS}, which says the following:
\begin{theorem}[Johnson-Lindenstrauss lemma]
  \label{thm:jl}
  Let $X \subset \R^n$ be a set of $N$ vectors, then for any
  $0 < \eps < 1/2$, there exists a map $f : X \to \R^m$ for
  some $m = \bigO(\eps^{-2} \lg N)$ such that
  \begin{equation*}
    \forall x, y \in X, (1 - \eps) \|x - y\|_2^2 \leq \|f(x) - f(y)\|_2^2 \leq (1 + \eps)\|x - y\|_2^2.
  \end{equation*}
\end{theorem}
This result dates back to 1984 and says that to preserve pairwise
Euclidean distances amongst a set of $N$ points/vectors in $\R^n$ to
within a factor $(1\pm \eps)$, it suffices to use just $m =
\bigO(\eps^{-2} \lg N)$ dimensions. The bound on $m$ was very recently
proven optimal~\cite{Larsen:2016:OotJLL}.

The standard technique for constructing a map with the properties of
Theorem~\ref{thm:jl} is the following: Let $A$ be an $m \times n$ matrix with
entries independently sampled as either $\mathcal{N}(0,1)$ random
variables (as in \cite{Dasgupta:2003:aEPoaToJaL}) or Rademacher
(uniform among $\{-1,+1\}$)  random variables
(as in \cite{Achlioptas:2003:DFRPJLwBC}). Once such entries have been
drawn, let
$f : \R^n \to \R^m$ be defined as:
\begin{equation*}
  f(x) = \frac{1}{\sqrt{m}} A x.
\end{equation*}
To prove that the map $f$ satisfies the guarantees in Theorem~\ref{thm:jl}, it is first shown
that for any vector $x$, the probability that $\|f(x)\|_2^2$ is not
within $(1\pm \eps)\|x\|_2^2$ is less than $1/N^2$. This probability
is called the error probability and denoted $\delta$. Using linearity
of $f$ and a union bound over all pairs $x, y \in X$, the
probability that all pairwise distances (i.e.\@ the norm of the vector
$x - y$) are preserved can be shown to be at least $1/2$.

\subsection{Time Complexity}
Examining the classic Johnson-Lindenstrauss reduction above, we see
that to embed a vector, we need to multiply with a dense matrix and
the embedding time becomes $\bigO(n m)$ (or equivalently
$\bigO(n \eps^{-2} \lg N)$). This may be prohibitively large for many
applications (recall one prime usage of dimensionality reduction is to
speed up algorithms), and much research has been devoted to obtaining
faster embedding time.

\paragraph*{Fast Johnson-Lindenstrauss Transform.}
Ailon and Chazelle~\cite{Ailon:2009:tFJLTaANN} were the first to
address the question of faster Johnson-Lindenstrauss transforms. In
their seminal paper, they introduced the so-called Fast
Johnson-Lindenstrauss transform for speeding up dimensionality
reduction. The basic idea in their paper is to first ``precondition''
the input data by multiplying with a diagonal matrix with random
signs, followed by multiplying with a Hadamard matrix. This has the
effect of ``spreading'' out the mass of the input vectors, allowing
for the dense matrix $A$ above to be replaced with a sparse
matrix. Since we can multiply with a Hadamard matrix using Fast
Fourier Transform, this gives an embedding time of
$\bigO(n \lg n + \eps^{-2} \lg^3 N)$ for embedding into the optimal
$m = \bigO(\eps^{-2} \lg N)$ dimensions. For
$m=\eps^{-2} \lg N \leq n^{1/2 - \gamma}$ for any constant
$\gamma > 0$, the embedding complexity was improved even further down
to $\bigO(n \lg m)$ in \cite{Ailon:2009:FDRURSoDBC}.

Another approach to achieve the $\bigO(n \lg m)$ embedding time, but
without the restriction on $\eps^{-2} \lg N \leq n^{1/2 - \gamma}$, is
to sacrifice the target dimension. This was done in
\cite{Ailon:2013:aAOUFJLT} and later improved in
\cite{Krahmer:2011:NaIJLEvtRIP}, where the embedding complexity was
$\bigO(n \lg m)$ at the cost of an increased target dimension $m =
\bigO(\eps^{-2} \lg N \lg^4 n)$.

\paragraph*{Sparse Vectors.}
Another approach to improve the performance of JL transforms, is to
assume the input data is sparse, i.e.\@ has few non-zero
coordinates. Designing an algorithm based on the work in
\cite{Weinberger:2009:FHfLSML}, Dasgupta et
al.~\cite{Dasgupta:2010:aSJLT} achieved an embedding complexity of
$\bigO(\|x\|_0 \eps^{-1} \lg^2 (mN) \lg N)$, where
$\|x\|_0 = |\{i \mid x_i \neq 0\}|$. This was later improved to
$\bigO(\|x\|_0 \eps^{-1} \lg N)$ in \cite{Kane:2014:SJLT}.

\paragraph*{Toeplitz Matrices.}
Finally, another very exciting approach is to use Toeplitz matrices or
partial circulant matrices for the embedding. We first introduce the
terminology.

An $m \times n$ Toeplitz matrix is an $m \times n$ matrix, where every
entry on a diagonal has the same value:
\begin{equation*}
  \begin{pmatrix}
    t_0 & t_1 & t_2 & \cdots & t_{n-1} \\
    t_{-1} & t_0 & t_1 & \cdots & t_{n-2} \\
    t_{-2} & t_{-1} & t_0 & \cdots & t_{n-3} \\
    \vdots & \vdots & \vdots & \ddots & \vdots \\
    t_{-(m-1)} & t_{-(m-2)} & t_{-(m-3)} & \cdots & t_{n - m}
  \end{pmatrix}
\end{equation*}
A partial circulant matrix is a special kind of Toeplitz matrix, where
every row, except the first, is the previous row rotated once:
\begin{equation*}
  \begin{pmatrix}
    t_0 & t_1 & t_2 & \cdots & t_{n-1} \\
    t_{n-1} & t_0 & t_1 & \cdots & t_{n-2} \\
    t_{n-2} & t_{n-1} & t_0 & \cdots & t_{n-3} \\
    \vdots & \vdots & \vdots & \ddots & \vdots \\
    t_{n-(m-1)} & t_{n-(m-2)} & t_{n-(m-3)} & \cdots & t_{n - m}
  \end{pmatrix}
\end{equation*}

Hinrichs and Vybíral \cite{Hinrichs:2011:JLLfCM} proposed the
following algorithm for generating a JL embedding based on a Toeplitz
matrix\footnote{\cite{Hinrichs:2011:JLLfCM} uses a partial circulant
  matrix but notes that a Toeplitz matrix could be used as well.}: Let
$t_{-(m-1)}, t_{-(m-2)}, \dotsc, t_{n-1}$ and $d_1, \dotsc, d_n$ be
i.i.d. Rademacher random variables, and $T$ be a Toeplitz matrix
defined from $t_{-(m-1)}, t_{-(m-2)}, \dotsc, t_{n-1}$ such that entry
$(i,j)$ takes values $t_{j-i}$ for $i =1, \dotsc, m$ and
$j=1, \dotsc, n$. Let $D$ be an $n \times n$ diagonal matrix with the
random variable $d_i$ giving the $i$'th diagonal entry. Define the map
$f$ as
\begin{equation*}
  f(x) = \frac{1}{\sqrt{m}} T D x.
\end{equation*}

Multiplying with a Toeplitz matrix corresponds to computing a
convolution and can be done using Fast Fourier Transform. By
appropriately blocking the input coordinates, the complexity of
embedding a vector $x$ is just $\bigO(n \lg m)$ for any target
dimension $m$. The big question is of course, how low can the target
dimension $m$ be, while preserving the distances between vectors up to
a factor of $1 \pm \eps$?

In the original paper \cite{Hinrichs:2011:JLLfCM}, the authors proved
that setting the target dimension to
$m = \bigO(\eps^{-2} \lg^3(1/\delta))$, the norm of any vector would
be preserved to within $(1 \pm \eps)$ with probability at least
$1 - \delta$. Setting $\delta = 1/N^2$, a union bound over all
pairwise difference vectors (as in the classic construction) shows
that dimension $m = \bigO(\eps^{-2} \lg^3 N)$ suffices. Later, the
analysis was refined in \cite{Vybiral:2011:aVotJLLfCM}, which lowered
the target dimension to $m = \bigO(\eps^{-2} \lg^2(1/\delta))$ for
preserving norms to within $(1 \pm \eps)$ with probability $1-\delta$.
Again, setting $\delta = 1/N^2$, this gives
$m = \bigO(\eps^{-2} \lg^2 N)$ target dimension. Now if the analysis
could be tightened even further to give the optimal
$m = \bigO(\eps^{-2} \lg N)$ dimensions, this would end the decades
long quest for faster and faster embedding algorithms!

\begin{table}[t]
  \centering
  \caption{Comparison of the performances of various
    Johnson-Lindenstrauss transform algorithms. $N$ is the number of
    input vectors, $n$ is the dimension of the input vectors, $m$ is
    the dimension of the output vectors, $\eps$ is the distortion.}
  \label{tab:jl_comparison}
  \begin{tabular}{l l l l l}
    \toprule
    Type & Embedding time & Target dimension ($m$) & Ref. & Notes \\
    \midrule
    Random projection & $\bigO(n m)$ & $\bigO(\eps^{-2} \lg N)$ & \cite{Dasgupta:2003:aEPoaToJaL} & \\
    Sparse & $\bigO(\|x\|_0 \eps^{-1} \lg^2 (mN) \lg N)$ & $\bigO(\eps^{-2} \lg N)$ & \cite{Dasgupta:2010:aSJLT} \\
    Sparse & $\bigO(\|x\|_0 \eps^{-1} \lg N)$ & $\bigO(\eps^{-2} \lg N)$ & \cite{Kane:2014:SJLT} \\
    FFT & $\bigO(n \lg n + m \lg^2 N)$ & $\bigO(\eps^{-2} \lg N) $ & \cite{Ailon:2009:tFJLTaANN} & \\
    FFT & $\bigO(n \lg m) $ & $\bigO(\eps^{-2} \lg N)$ & \cite{Ailon:2009:FDRURSoDBC} & $m \leq n^{1/2 - \gamma}$ \\
    FFT & $\bigO(n \lg m) $ & $\bigO(\eps^{-2} \lg N \lg^4 n)$ & \cite{Krahmer:2011:NaIJLEvtRIP} \\
    Toeplitz & $\bigO(n \lg m)$ & $\bigO(\eps^{-2} \lg^3 N)$ & \cite{Hinrichs:2011:JLLfCM} \\
    Toeplitz & $\bigO(n \lg m)$ & $\bigO(\eps^{-2} \lg^2 N)$ & \cite{Vybiral:2011:aVotJLLfCM} \\
    \bottomrule
  \end{tabular}
\end{table}

\paragraph*{Our Contribution.}
Our main result unfortunately shows that the analysis of
Vybíral~\cite{Vybiral:2011:aVotJLLfCM} cannot be tightened to give an
even lower target dimensionality. More specifically, we prove that the
upper bound given in \cite{Vybiral:2011:aVotJLLfCM} is optimal:
\begin{theorem}
  \label{thm:highlevelmain}
  Let $T$ and $D$ be the $m \times n$ Toeplitz and $n \times n$ diagonal matrix in the embedding
  proposed by~\cite{Hinrichs:2011:JLLfCM}. For all $0 < \eps < C$, where $C$ is a universal
  constant, and any desired error probability $\delta > 0$, if the
  following holds for every unit vector $x \in \R^n$:
  \begin{equation*}
    \Pr\left[\left|\left\|\frac{1}{\sqrt{m}}TDx\right\|_2^2 - 1\right| <
      \eps\right] > 1-\delta,
  \end{equation*}
  then it must be the case that
  $m = \Omega(\eps^{-2} \lg^2(1/\delta))$.
\end{theorem}

While Theorem~\ref{thm:highlevelmain} already shows that one cannot
tighten the analysis of Vybíral for preserving the norm of just one
vector, Theorem~\ref{thm:highlevelmain} does leave open the
possibility that one would not need to union bound over all $N^2$
pairs of difference vectors when trying to preserve all pairwise
distances amongst a set of $N$ vectors. It could still be the case
that there somehow was a strong positive correlation between distances
being preserved (though this seems extremely unlikely, and would be
something not seen in any previous approach to JL). To complete the
picture, we indeed show that this is not the case, at least for $N$
somewhat smaller than the dimension $n$:
\begin{theorem}
  \label{thm:highlevelmainMany}
  Let $T$ and $D$ be the $m \times n$ Toeplitz and $n \times n$
  diagonal matrix in the embedding proposed
  by~\cite{Hinrichs:2011:JLLfCM}. For all $0 < \eps < C$, where $C$ is
  a universal constant, if the following holds for every set of $N$
  vectors $X \subset \R^n$:
  \begin{equation*}
    \Pr\left[\forall x,y \in X : \left|\left\|\frac{1}{\sqrt{m}}TDx -
          \frac{1}{\sqrt{m}}TDy\right\|_2^2  - \| x-y\|_2^2\right|  \leq
      \eps\|x-y\|_2^2 \right] = \Omega(1),
  \end{equation*} then it must be the case
that either $m = \Omega(\eps^{-2} \lg^2 N)$ or $m = \Omega(n/N)$.
\end{theorem}

We remark that our proofs also work if we replace $T$ be a partial
circulant matrix (which was also proposed
in~\cite{Hinrichs:2011:JLLfCM}). Furthermore, we expect that minor
technical manipulations to our proof would also show the above
theorems when the entries of $T$ and $D$ are $\mathcal{N}(0,1)$
distributed rather than Rademacher (this was also proposed
in~\cite{Hinrichs:2011:JLLfCM}).

\section{Lower Bound for One Vector}
\label{sec:onevector}
Let $T$ be $m \times n$ Toeplitz matrix defined from random variables
$t_{-(m-1)},t_{-(m-2)}, \dots, t_{n-1}$ such that entry
$(i,j)$ takes values $t_{j-i}$ for $i =1,\dots,m$ and
$j=1,\dots,n$. Let $D$ be an $n \times n$ diagonal matrix with the
random variable $d_i$ giving the $i$'th diagonal entry. This section
shows the following:
\begin{theorem}\label{thm:main}
  Let $T$ be $m \times n$ Toeplitz and $D$ $n \times n$ diagonal. If
  $t_{-(m-1)},t_{-(m-2)}, \dots, t_{n-1}$ and
  $d_1,\dots,d_n$ are independently distributed Rademacher random
  variables for $i=-(m-1),\dots,n-1$ and $j=1,\dots,n$, then for all
  $0 < \eps < C$, where $C$ is a universal constant, there exists a
  unit vector $x \in \R^n$ such that
  \begin{equation*}
    \Pr\left[\left|\left\|\frac{1}{\sqrt{m}}TDx\right\|_2^2 - 1\right| > \eps\right] \geq 2^{-\bigO(\eps \sqrt{m})}.
  \end{equation*}
  and furthermore, all but the first $O(\sqrt{m})$ coordinates of $x$
  are $0$.
\end{theorem}
It follows from Theorem~\ref{thm:main} that if we want to have probability at
least $1-\delta$ of preserving the norm of \emph{any} unit vector $x$
to within $(1 \pm \eps)$, it must be the case that
$\eps \sqrt{m} = \Omega(\lg(1/\delta))$, i.e.\@
$m = \Omega(\eps^{-2} \lg^2(1/\delta))$. This is precisely the
statement of Theorem~\ref{thm:highlevelmain}. Thus we set out to prove
Theorem~\ref{thm:main}.

To prove Theorem~\ref{thm:main}, we wish to invoke the Paley-Zygmund
inequality, which states, that if $X$ is a non-negative random
variable with finite variance and $0 \leq \theta \leq 1$, then
\begin{equation*}
  \Pr [X > \theta \E[X]] \geq (1 - \theta)^2 \frac{\E^2[X]}{\E[X^2]}.
\end{equation*}

We carefully choose a unit vector $x$, and define the random variable
for Paley-Zygmund to be the $k$'th moment of the difference between
the norm of
$x$ transformed and 1.
\begin{proof}
  Let $k$ be an even positive integer less than $m/4$ and define
  $s := 4k$. Note that $s \leq m$. Let $x$ be an arbitrary
  $n$-dimensional unit vector such that the first $s$ coordinates are
  in $\{-1/\sqrt{s}, +1/\sqrt{s}\}$, while the remaining $n - s$
  coordinates are 0. Define the random variable parameterized by $k$
  \begin{equation*}
    Z_k := \left(\left\|\frac{1}{\sqrt{m}} TDx\right\|_2^2 - 1 \right)^k.
  \end{equation*}
  Since $k$ is even, the random variable $Z_k$ is non-negative.

  We wish to lower bound $\E[Z_k]$ and upper bound $\E[Z_k^2]$ in order to
  invoke Paley-Zygmund. The bounds we prove are as follows:
  \begin{lemma}
    \label{lem:boundZ}
    If $k \leq \sqrt{m}$, then the random variable $Z_k$ satisfies:
    \begin{flalign*}
      && \E[Z_k] &\geq m^{-k/2} k^k 2^{-\bigO(k)} & \\
      \text{and} && \E[Z_k^2] &\leq m^{-k} k^{2k} 2^{\bigO(k)}. &
    \end{flalign*}
  \end{lemma}

  Before proving Lemma~\ref{lem:boundZ} we show how to use it together
  with Paley-Zygmund to complete the proof of Theorem~\ref{thm:main}.

  We start by invoking Paley-Zygmund and then rewriting the expectations
  according to Lemma~\ref{lem:boundZ},
  \begin{align*}
    \Pr[Z_k > \E[Z_k]/2] &\geq (1/4)\frac{\E^2[Z_k]}{\E[Z_k^2]} \implies \\
    \Pr[Z_k^{1/k} > (\E[Z_k]/2)^{1/k}] &\geq (1/4)\frac{\E^2[Z_k]}{\E[Z_k^2]} \implies \\
    \Pr\left[\left|\left\|\frac{1}{\sqrt{m}} TDx\right\|_2^2 - 1 \right| >
    \frac{k}{C_0 \sqrt{m}}\right] &\geq 2^{-\bigO(k)}.
  \end{align*}
  Here $C_0$ is some constant greater than $0$.  For any $0<\eps<1/C_0$,
  we can now set $k$ such that $k/(C_0 \sqrt{m}) = \eps$, i.e.\@ we choose
  $k = \eps C_0 \sqrt{m}$. This choice of $k$ satisfies
  $k \leq \sqrt{m}$ as required by Lemma~\ref{lem:boundZ}. We have thus shown
  that:
  \begin{align*}
    \Pr\left[\left|\left\|\frac{1}{\sqrt{m}} TDx\right\|_2^2 - 1 \right| >
    \eps \right] &\geq 2^{-\bigO(\eps \sqrt{m})}.
  \end{align*}
\end{proof}

\begin{remark}
  Theorem~\ref{thm:main} can easily be extended to partial circulant
  matrices.  The difference between partial circulant and Toeplitz
  matrices is the dependence between the values in the first $m$ and
  last $m$ columns. However, as only the first
  $s = 4k \leq 4 \sqrt{m}$ entries in $x$ are nonzero, the last $m$
  columns are ignored, and so partial circulant and Toeplitz matrices
  behave identically in our proof.
\end{remark}

\begin{proof}[Proof of Lemma~\ref{lem:boundZ}]
  Before we prove the two bounds in Lemma~\ref{lem:boundZ} individually, we
  rewrite $\E[Z_k]$, as this benefits both proofs.
  \begin{align*}
    \E[Z_k] &= \E\left[\left(\left\|\frac{1}{\sqrt{m}} TDx\right\|_2^2 - 1 \right)^k\right] \\
            &= \E\left[\left(\left(\frac{1}{m}\sum_{i=1}^m \left(\sum_{j=1}^n t_{j-i} d_j x_j\right)^2 \right)- 1 \right)^k\right] \\
            &= \E\left[\left(\left(\frac{1}{m}\sum_{i=1}^m \left(\left(\sum_{j=1}^n t^2_{j-i} d^2_j x^2_j\right) + \left(\sum_{j=1}^n \sum_{h \in \{1,\dots,n\} \setminus \{j\}} t_{j-i}t_{h-i} d_j d_h x_j x_h  \right) \right)\right)- 1 \right)^k\right] \\
            &= \E\left[\left(\frac{1}{m}\sum_{i=1}^m \left(\left(\sum_{j=1}^n t^2_{j-i} d^2_j x^2_j - x_j^2\right) + \left(\sum_{j=1}^n \sum_{h \in \{1,\dots,n\} \setminus \{j\}} t_{j-i}t_{h-i} d_j d_h x_j x_h  \right) \right)\right)^k\right] \\
            &= \E\left[\left(\frac{1}{m}\sum_{i=1}^m \sum_{j=1}^n \sum_{h \in \{1,\dots,n\} \setminus \{j\}} t_{j-i}t_{h-i} d_j d_h x_j x_h \right)^k\right] \\
            &= \frac{1}{m^{k}} \sum_{S \in ([m] \times [n] \times [n])^k \mid \forall (i,j,h) \in S : h \neq j} \E\left[ \prod_{(i,j,h) \in S} t_{j-i}t_{h-i}d_j d_h x_j x_h \right]
  \end{align*}
  Observe that for $j > s$ or $h > s$ the product becomes $0$, as
  either $x_j$ or $x_h$ is 0. By removing all these terms, we simplify the sum to
  \begin{align*}
    \E[Z_k] &= \frac{1}{m^{k}} \sum_{S \in ([m] \times [s] \times [s])^k \mid \forall (i,j,h) \in S : h \neq j} \E\left[ \prod_{(i,j,h) \in S} t_{j-i}t_{h-i}d_j d_h x_j x_h \right]
  \end{align*}

  Observe for an $S \in ([m] \times [s] \times [s])^k$, that the value
  $\E\left[ \prod_{(i,j,h) \in S} t_{j-i}t_{h-i}d_j d_h x_j x_h \right]$
  is $0$ if one of the following two things are true:
  \begin{itemize}
  \item A $d_j$ occurs an odd number of times in the product.
  \item A variable $t_a$ occurs an odd number of times in the product.
  \end{itemize}
  To see this, note that by the independence of the random variables, we
  can write the expectation of the product, as a product of expectations
  where each term in the product has all the occurrences of the same
  random variable. Since the $d_j$'s and $t_a$'s are Rademachers, the
  expectation of any odd power of one of these random variables is
  $0$. Thus if just a single random variable amongst the $d_j$'s and
  $t_a$'s occurs an odd number of times, we have
  $\E\left[ \prod_{(i,j,h) \in S} t_{j-i}t_{h-i}d_j d_h x_j x_h
  \right]=0$. Similarly, we observe that if every random variable occurs
  an even number of times, then the expectation of the product is
  exactly $1/s^k$ since each $x_j$ also occurs an even number of
  times. If $\Gamma_k$ denotes the number of tuples
  $S \in ([m] \times [s] \times [s])^k$ such that
  $\forall (i,j,h) \in S$ we have $h \neq j$ and furthermore:
  \begin{itemize}
  \item For all columns $a \in [s]$, $|\{(i,j,h) \in S \mid  j=a \vee h=a\}| \mod 2 = 0$.
  \item For all diagonals $a \in \{-(m-1),\dots,s-1\}$, $|\{(i,j,h) \in S \mid  j-i=a \vee h-i=a\}| \mod 2 = 0$.
  \end{itemize}
  Then we conclude
  \begin{equation}
    \label{eq:e_zk_rewrite}
    \E[Z_k] = \frac{\Gamma_k}{s^k m^k}.
  \end{equation}

  Note that $Z_k^2 = Z_{2k}$. Therefore,
  \begin{equation}
    \label{eq:e_z2k_rewrite}
    \E[Z_k^2] = \E[Z_{2k}] = \frac{\Gamma_{2k}}{s^{2k} m^{2k}}.
  \end{equation}

  To complete the proof of Lemma~\ref{lem:boundZ} we need lower and upper
  bounds for $\Gamma_k$ and $\Gamma_{2k}$. The bounds we prove are
  \begin{lemma}
    \label{lem:boundGamma}
    If $k \leq \sqrt{m}$, then $\Gamma_k$ and $\Gamma_{2k}$ satisfy:
    \begin{flalign*}
      && \Gamma_k &= m^{k/2} s^k k^k 2^{-\bigO(k)} & \\
      \text{and} && \Gamma_{2k} &= m^{k} s^{2k} k^{2k} 2^{\bigO(k)}. &
    \end{flalign*}
  \end{lemma}

  The proofs of the two bounds in Lemma~\ref{lem:boundGamma} are in
  Sections \ref{sec:lower-bound-t} and \ref{sec:upper-bound-t}.

  Substituting the bounds from Lemma~\ref{lem:boundGamma} in
  \eqref{eq:e_zk_rewrite} and \eqref{eq:e_z2k_rewrite} we get
  \begin{align*}
    \E[Z_k] &= m^{-k/2} k^k 2^{-\bigO(k)} \\
    \E[Z_k^2] &= m^{-k} k^{2k} 2^{-\bigO(k)},
  \end{align*}
  which are the bounds we sought for Lemma~\ref{lem:boundZ}.
\end{proof}

\subsection{Lower Bounding $\Gamma_k$}
\label{sec:lower-bound-t}

We first recall that the definition of $\Gamma_k$ is the number of tuples
$S \in ([m] \times [s] \times [s])^k$ satisfying that
$\forall (i,j,h) \in S$ we have $h \neq j$ and furthermore:
\begin{itemize}
\item For all columns $a \in [s]$,
  $|\{(i,j,h) \in S \mid j=a \vee h=a\}| \mod 2 = 0$.
\item For all diagonals $a \in \{-(m-1),\dots,s-1\}$,
  $|\{(i,j,h) \in S \mid j-i=a \vee h-i=a\}| \mod 2 = 0$.
\end{itemize}

We view a triple $(i, j, h) \in ([m] \times [s] \times [s])$ as two
entries $(i, j)$ and $(i, h)$ in an $m \times s$ matrix. Furthermore,
when we say that a triple \emph{touches} a column or diagonal, a
matrix entry of the triple lie on that column or diagonal, so
$(i, j, h)$ touches columns $j$ and $h$ and diagonals $j - i$ and
$h - i$. Similarly, we say that a tuple
$S \in ([m] \times [s] \times [s])^k$ touches a given column or
diagonal $l$ times, if $l$ triples in $S$ touches that column or
diagonal.

We intent to prove a lower bound for $\Gamma_k$ by constructing a big
family of tuples
$\mathcal{F} \subseteq ([m] \times [s] \times [s])^k$, where each
tuple satisfies, that each column and diagonal touched by that tuple
is touched exactly twice. As each column and diagonal is touched an
even number of times, the number of tuples in the family is a lower
bound for $\Gamma_k$.
\begin{proof}[Proof of $\Gamma_k = m^{k/2} s^k k^k 2^{-\bigO(k)}$]
  We describe how to construct a family of tuples
  $\mathcal{F} \subseteq ([m] \times [s] \times [s])^k$ satisfying
  that $\forall S \in \mathcal{F}, \forall (i,j,h) \in S$ we have
  $h \neq j$ and furthermore:
  \begin{itemize}
  \item For all columns $a \in [s]$,
    $|\{(i,j,h) \in S \mid j=a \vee h=a\}| \in \{0, 2\}$.
  \item For all diagonals $a \in \{-(m-1),\dots,s-1\}$,
    $|\{(i,j,h) \in S \mid j-i=a \vee h-i=a\}| \in \{0, 2\}$.
  \end{itemize}

  From this and the definition of $\Gamma_k$ it is clear that
  $|\mathcal{F}| \leq \Gamma_k$.

  When constructing an $S \in \mathcal{F}$, we view $S$ as consisting
  of two halves $S_1$ and $S_2$, such that $S_1$ touches exactly the
  same columns and diagonals as $S_2$ and both $S_1$ and $S_2$ touches
  each column and diagonal at most once. To capture this, we give the
  following definition, where $\mathbb{S}$ is meant to be the family
  of such halves $S_1$ and $S_2$.
  \begin{definition}
    \label{def:bbS}
    Let $\mathbb{S}$ be the set of all tuples
    $S \in ([m] \times [s] \times [s])^{k/2}$ such that
    \begin{itemize}
    \item $\forall (i, j, h) \in S, j \neq h$
    \item For all columns
      $a \in [s], |\{(i, j, h) \in S \mid j = a \vee h = a\}| \leq 1$
    \item For all diagonals
      $a \in \{-(m-1), \dotsc, s - 1\}, |\{(i, j, h) \in S \mid j - i
      = a \vee h - i = a\}| \leq 1$
    \end{itemize}
  \end{definition}

  Definition~\ref{def:bbS} mimics the definition of $\Gamma_k$, and the first
  item in Definition~\ref{def:bbS} ensures that the triples in a tuple in
  $\mathbb{S}$ are of the same form as in $\Gamma_k$. The final two
  items ensure that each column and diagonal, respectively, is touched
  at most once. This is exactly the properties we wanted of $S_1$ and
  $S_2$ individually.

  We can now construct $\mathcal{F}$ as all pairs of (half) tuples
  $S_1, S_2 \in \mathbb{S}$, such that $S_1$ touches exactly the same
  columns and diagonals as $S_2$. To capture that $S_1$ and $S_2$ touch
  the same columns and diagonals, we introduce the notion of a
  signature. A signature of $S_i$ is the set of columns and diagonals
  touched by $S_i$.

  To have $S_1$ and $S_2$ touch exactly the same columns and
  diagonals, it is necessary and sufficient that they have the same
  signature.

  We introduce the following notation: $B$ denotes the number of
  signatures with at least one member, and by enumerating the
  signatures, we let $b_i$ denote the number of (half) tuples in
  $\mathbb{S}$ with signature $i$.

  We recall that a (half) tuple $S_1 \in \mathbb{S}$ touches each
  column and diagonal at most once, and if $S_1$ and $S_2$ share the
  same signature, they touch exactly the same columns and
  diagonals. Therefore, using $\circ$ to mean concatenation,
  $S = S_1 \circ S_2 \in \mathcal{F}$, as each column and diagonal
  touched is touched exactly twice. Therefore $|\mathcal{F}|$ is a
  lower bound for $\Gamma_k$. Note that for a given signature $i$, the
  number of choices of $S_1$ and $S_2$ with that signature is
  $b_i^2$. This gives the following inequality,
  \begin{equation*}
    \Gamma_k \geq |\mathcal{F}| = \sum_{i=1}^B b_i^2.
  \end{equation*}

  We now apply the Cauchy-Schwarz inequality:
  \begin{equation}
    \sum_{i=1}^B b_i^2 \sum_{i=1}^B 1^2 \geq \big(\sum_{i=1}^B b_i \big)^2 \implies
    \sum_{i=1}^B b_i^2 \geq \frac{\big(\sum_{i=1}^B b_i \big)^2}{\sum_{i=1}^B 1^2} \implies
    \label{eq:lower_bound_sum_bi2}
    \Gamma_k \geq \frac{|\mathbb{S}|^2}{B}.
  \end{equation}

  To get a lower bound on $|\mathbb{S}|^2/B$ (and in turn $\Gamma_k$),
  we need a lower bound on $|\mathbb{S}|$ and an upper bound on
  $B$. These bounds are stated in the following lemmas
  \begin{lemma}
    \label{lem:select_distinct_mnn}
    $|\mathbb{S}| = \Omega(m^{k/2} s^k 2^{-k})$.
  \end{lemma}
  \begin{lemma}
    \label{lem:num_sigs}
    $ B = \bigO\left(\binom{m + s}{k/2} s^{k/2} \binom{s}{k} \right) $
  \end{lemma}

  Before proving any of these lemmas, we show that they together with
  \eqref{eq:lower_bound_sum_bi2} give the desired lower bound on
  $\Gamma_k$:
  \begin{align}
    \Gamma_k &= \frac
          {\Omega(m^{k/2} s^k 2^{-k})^2}
          {\bigO\left(\binom{m + s}{k/2} s^{k/2} \binom{s}{k} \right)}
    \label{eq:gamma_lower_ugly}
             = \Omega\left(\frac
          {m^k s^{2k} 2^{-2k} (k/2)^{k/2} k^k}
          {(m + s)^{k/2} s^{k/2} s^k}
          \right).
  \end{align}

  Because $s = 4k$, we have
  $\frac{(k/2)^{k/2}}{s^{k/2}} = 2^{-\Theta(k)}$, and because
  $s \leq m$: $\frac{m^k}{(m+s)^{(k/2)}} = m^{k/2}
  2^{-\Theta(k)}$. With this we can simplify
  \eqref{eq:gamma_lower_ugly},
  \begin{align*}
    \Gamma_k &= m^{k/2} s^k k^k 2^{-\bigO(k)}.
  \end{align*}
  which is the lower bound we sought.
\end{proof}

\begin{proof}[Proof of Lemma~\ref{lem:select_distinct_mnn}]
  Recall that $\mathbb{S} \subseteq ([m] \times [s] \times [s])^{k/2}$
  is the set of (half) tuples that touch each column and diagonal at
  most once, and, for each triple $(i, j, h)$ in these (half)
  tuples, we have $j \neq h$.

  We prove Lemma~\ref{lem:select_distinct_mnn} by analysing how we can
  create a large number of distinct $S \in \mathbb{S}$ by choosing the
  triples in $S$ iteratively.

  For each triple, we choose a row and two distinct entries on this
  row. We choose the row among any of the $m$ rows.

  However, because $S \in \mathbb{S}$, when choosing entries on the
  row, we cannot choose entries that lie on columns or diagonals
  touched by previously chosen triples. Instead we choose the two
  entries among any of the other entries. Therefore, whenever we
  choose a triple, this triple prevents at most four row entries from
  being chosen for every subsequent triple, as the two diagonals and
  two columns touched by the chosen triple intersect with at most four
  entries on the rows of the subsequent triples. This leads to the
  following recurrence, describing a lower bound for the number of
  triples
  \begin{align}
    \label{eq:select_distinct_mnn:recurr}
    F(r, c, t) =
    \begin{cases}
      r \cdot c \cdot (c - 1) \cdot F(r, c - 4, t - 1) & \text{if } t > 0 \\
      1 & \text{otherwise}
    \end{cases}
  \end{align}
  where $r$ is the number of rows to choose from, $c$ is the minimum
  number of choosable entries in any row, and $t$ is the number of
  triples left to choose.

  Inspecting \eqref{eq:select_distinct_mnn:recurr}, we can see that
  $F$ can equivalently be defined as
  \begin{align}
    \label{eq:select_distinct_mnn:recurr_simpl}
    F(r, c, t) &= r^t \prod_{i = 0}^{t-1} (c - 4i) (c - 1 - 4i).
  \end{align}

  If $t \leq \frac{c}{8}$ then the terms inside the product in
  \eqref{eq:select_distinct_mnn:recurr_simpl} are greater than
  $\frac{c}{2}$, so we can bound $F$ from below:
  \begin{align*}
    F(r, c, t) &\geq r^t \big( \frac{c}{2} \big)^{2t} = r^t c^{2t} \frac{1}{4^t}.
  \end{align*}

  We now insert the values for $r, c$ and $t$ to find a lower bound
  for $|\mathbb{S}|$, noting that $s = 4 k$ ensures that
  $t \leq \frac{c}{8}$:
  \begin{equation*}
    |\mathbb{S}| \geq F \big( m, s, \frac{k}{2} \big) \geq m^{k/2} s^k \frac{1}{4^{k/2}} \implies
    |\mathbb{S}| = \Omega(m^{k/2} s^k 2^{-k}).
  \end{equation*}
\end{proof}

\begin{proof}[Proof of Lemma~\ref{lem:num_sigs}]
  Recall that for at triple $S \in \mathbb{S}$ we define the signature
  as the set of columns and diagonals touched by $S$. Furthermore,
  viewing a triple $(i, j, h) \in ([m] \times [s] \times [s])$ as the
  two entries $(i, j)$ and $(i, h)$ in an $m \times s$ matrix, we
  define the left endpoint as $(i, \min\{j, h \}$ and the right
  endpoint as $(i, \max\{j, h\})$.

  The claim to prove is
  \begin{equation*}
    B = \bigO\left(\binom{m + s}{k/2} s^{k/2} \binom{s}{k} \right).
  \end{equation*}
  This is proven by first showing an upper bound on the number of
  choices for the diagonals of left endpoints, then diagonals of right
  endpoints and finally for columns.

  In an $m \times s$ matrix there are $m + s$ different diagonals and
  as the chosen diagonals have to be distinct, there are
  $\binom{m + s}{k/2}$ choices for the diagonals corresponding to
  left endpoints in a triple.

  As the right endpoint of a triple has to be in the same row as the
  left endpoint, there are at most $s$ choices for the diagonal
  corresponding to the right endpoint when the left endpoint has been
  chosen (which it has in our case). This gives a total of $s^{k/2}$
  choices for diagonals corresponding to right endpoints.

  Finally, there are $s$ columns to choose from and the chosen columns
  have to be distinct, and so the total number of choices of columns
  is $\binom{s}{k}$.

  The product of these number of choices gives the upper bound sought.
\end{proof}

\subsection{Upper Bounding $\Gamma_{2k}$}
\label{sec:upper-bound-t}

\begin{proof}
  Recall that $\Gamma_{2k}$ is defined as the number of tuples
  $S \in ([m] \times [s] \times [s])^{2k}$ such that
  $\forall (i,j,h) \in S$ we have $h \neq j$ and furthermore:
  \begin{itemize}
  \item For all columns $a \in [s]$, $|\{(i,j,h) \in S \mid  j=a \vee h=a\}| \mod 2 = 0$.
  \item For all diagonals $a \in \{-(m-1),\dots,s-1\}$, $|\{(i,j,h) \in S \mid  j-i=a \vee h-i=a\}| \mod 2 = 0$.
  \end{itemize}

  Let $\mathcal{F} \subseteq ([m] \times [s] \times [s])^{2k}$ be the
  family of tuples satisfying these conditions, and so
  $|\mathcal{F}| = \Gamma_{2k}$.

  To prove an upper bound on $\Gamma_{2k}$, we show how to encode a
  tuple $S \in \mathcal{F}$ using at most
  $k \lg m + 2k \lg s + 2k \lg k + \bigO(k)$ bits, such that $S$ can
  be decoded from this encoding. Since any $S \in \mathcal{F}$ can be
  encoded using $k \lg m + 2k \lg s + 2k \lg k + \bigO(k)$ bits and
  $|\mathcal{F}| = \Gamma_{2k}$, we can conclude:
  \begin{equation*}
    \Gamma_{2k} = 2^{k \lg m + 2k \lg s + 2k \lg k + \bigO(k)}
    = m^{k} s^{2k} k^{2k} 2^{\bigO(k)}.
  \end{equation*}

  Let $\sigma$ denote the encoding function and $\sigma^{-1}$ denote
  the decoding function. If $S \in \mathcal{F}$ and $t \in S$,
  $\sigma(t)$ denotes the encoding of the triple $t$, $\sigma(S)$
  denotes the encoding of the entire tuple $S$, and
  $\sigma(\mathcal{F})$ denotes the image of $\sigma$.

  A tuple $S \in \mathcal{F}$ consists of triples
  $t_1, t_2, \dotsc, t_{2k}$ such that
  $S = t_1 \circ t_2 \circ \dotsb \circ t_{2k}$. To encode
  $S \in \mathcal{F}$ we encode each of the triples and store them in
  the same order:
  $\sigma(S) = \sigma(t_1) \circ \sigma(t_2) \circ \dotsb \circ
  \sigma(t_{2k})$.

  We will first describe a graph view of a tuple $S$ which will be
  useful for encoding and decoding, then we will show an encoding
  algorithm and finally a decoding algorithm.

  \paragraph*{Graph}

  A tuple $S \in \mathcal{F}$ forms a (multi)graph structure, where
  every triple $(i, j, h) \in S$ is a vertex. Since
  $S \in \mathcal{F}$, there lies an even number of triple endpoints
  on each diagonal. We can thus pair endpoints lying on the same
  diagonal, such that every endpoint is paired with exactly one other
  endpoint. When two triples have endpoints that are paired, the
  triples have an edge between them in the graph. As every triple has
  two endpoints, every vertex has degree two, and so the graph
  consists entirely of simple cycles of length at least two.

  \paragraph*{Encoding}

  To encode an $S \in \mathcal{F}$, we first encode each cycle by
  itself by defining the $\sigma(t)$'s for the triples $t$ in the
  cycle. After this, we order the defined $\sigma(t)$'s as the $t$'s
  were ordered in the input.
  \begin{enumerate}
  \item For each cycle we perform the following:
    \begin{enumerate}
    \item We pick any vertex $t = (i, j, h)$ of the cycle and give it
      the type \textsf{head}. Define $\sigma(t)$ as the concatenation
      of its type \textsf{head}, its row $i$, and two columns $j$ and
      $h$. This uses $\lg m + 2 \lg s + \bigO(1)$ bits.
    \item We iterate through the cycle starting after \textsf{head}
      and give vertices, except the last, type \textsf{mid}. The last
      vertex just before \textsf{head} is given the type
      \textsf{last}.
    \item For each triple $t$ of type \textsf{mid} we store its type
      and two columns explicitly. However, instead of storing its row
      we store the index of its predecessor in the cycle order as well
      as how they are connected: If we typed $t_r$ just before $t_s$
      when iterating through the cycle, when encoding $t_s$ we store
      $r$ as well as whether $t_r$ and $t_s$ are connected by the left
      or right endpoint in $t_r$ and left or right endpoint in
      $t_s$. So define $\sigma(t)$ as the concatenation of
      \textsf{mid}, the two columns, the predecessor index and how it
      is connected to the predecessor. All in all we spend
      $\lg k + 2 \lg s + \bigO(1)$ bits encoding each \textsf{mid}.
    \item Finally, to encode the triple $t$, which is typed
      \textsf{last}, we define $\sigma(t)$ as the concatenation of its
      type, its predecessors index, how it is connected to its
      predecessor, and the column of the endpoint on the predecessor's
      diagonal. We thus spend $\lg k + \lg s + \bigO(1)$ bits to
      encode a \textsf{last}. However, since $s = 4k$ the number of
      bits per encoded \textsf{last} is equivalent to
      $2 \lg k + \bigO(1)$, which turns out to simplify the analysis
      later.
    \end{enumerate}

    Note that for each triple, the type is encoded in the first 2
    bits\footnote{To encode the type in 2 bits we could use the
      following scheme: $00 = \mathsf{head}$, $01 = \mathsf{mid}$,
      $10 = \mathsf{last}$, and $11$ is unused.}  in the encoding of
    the triple. This will be important during decoding.

  \item After encoding all cycles, we order the encoded triples in the
    same order as the triples in the input, and output the
  concatenation of the encoded triples:
    $\sigma(t_1) \circ \sigma(t_2) \circ \dotsb \circ \sigma(t_{2k})$.
  \end{enumerate}

  To analyse the number of bits needed in total, we look at the
  average number of bits per triple inside a cycle. Since all cycles
  have a length of at least two, each cycle has a \textsf{head} and a
  \textsf{last} triple. These two use an average of
  $\frac{\lg m}{2} + \lg s + \lg k + \bigO(1)$ bits. We now claim that
  the number of bits per \textsf{mid} is bounded by this average.
  Recall that we assumed $\sqrt{m} \geq k$ and $s = 4k$:
  \begin{alignat*}{4}
    \sqrt{m} \geq k \implies \frac{\lg m}{2} \geq \lg k \implies \frac{\lg m}{2} + \lg s + \lg k + 2 &\geq \lg k + 2 \lg s \implies \\
    \frac{\lg m}{2} + \lg s + \lg k + \bigO(1) &\geq |\mathsf{mid}|.
  \end{alignat*}

  Since the average of all $2k$ triples is at most
  $\frac{\lg m}{2} + \lg s + \lg k + \bigO(1)$ bits, the number of
  bits for all triples is at most
  $k \lg m + 2k \lg s + 2k \lg k + \bigO(k)$.

  \paragraph*{Decoding}
  We now show how to decode any $\sigma(S) \in \sigma(\mathcal{F})$
  such that $\sigma^{-1}(\sigma(S)) = S$.

  First we need to extract the encodings of the individual triples,
  then we decode each cycle, and finally we restore the order of
  triples. The following steps describes this algorithm in more
  detail.

  \begin{enumerate}
  \item We first extract the individual $\sigma(t_i)$'s, by iterating
    through the bit string $\sigma(S)$. By looking at the type in the
    first two bits of an encoded triple we know the length of the
    encoded triple. From this we can extract $\sigma(t_i)$ as well as
    know where $\sigma(t_{i+1})$
    begins. \label{item:decoding_extract_encoded_triples}
  \item We now wish to decode each cycle to get the row and two columns
    of every triple, so do the following for each \textsf{head} $t_i$:
    \begin{enumerate}
    \item $t \leftarrow t_i$
    \item Look at the first two bits in $t$ to determine its type, and
      while $t$ is not a \textsf{last}, do:
      \begin{enumerate}
      \item If $t$ is a \textsf{head}, the row and two columns are
        stored explicitly.
      \item If $t$ is a \textsf{mid}, the two columns are stored
        explicitly. From the reference to $t$'s predecessor, we can
        calculate the diagonal shared between $t$ and its
        predecessor. This is can be done as we have stored which
        endpoint (left or right) of the predecessor lies on the
        diagonal, and as we decode in the same order as we encoded, we
        have already decoded the predecessor.

        We know which of $t$'s columns the shared diagonal intersects,
        and so we can calculate the row.

        If the predecessor is a \textsf{head}, take note of which
        diagonal is shared with it.
      \item $t \leftarrow$ the triple that has $t$ as its predecessor.
      \end{enumerate}
    \item $t$ is now a \textsf{last}, so we know the predecessor's
      index as well as a column. However, as the graph consists of
      cycles rather than just linked lists, \textsf{last} shares a
      diagonal with its predecessor as well as sharing a diagonal with
      \textsf{head}. We have noted which diagonal has already been
      used for \textsf{head}, so the other diagonal \textsf{head}
      touches is shared with \textsf{last}.

      From these two diagonals and the column, it is possible to
      calculate the row by the intersection between the predecessor
      diagonal and the column, and the other column by the
      intersection of the row and the other diagonal.
    \end{enumerate}
  \item Finally order the triples as they were, when the encoded
    versions were extracted in step
    \ref{item:decoding_extract_encoded_triples}.
  \end{enumerate}
\end{proof}

\section{Lower Bound for $N$ Vectors}
\label{sec:lower-bound-n-vectors}
In this section, we generalize the result of Section~\ref{sec:onevector} to
obtain a lower bound for preserving all pairwise distances amongst a
set of $N$ vectors. Our proof uses Theorem~\ref{thm:main} as a building
block. Recall that Theorem~\ref{thm:main} guarantees that there is a vector
$x$ such that
\begin{equation*}
  \Pr\left[\left|\left\|\frac{1}{\sqrt{m}}TDx\right\|_2^2 - 1\right| > \eps\right] \geq 2^{-\bigO(\eps \sqrt{m})}.
\end{equation*}
Moreover, the vector $x$ has non-zeroes only in the first
$O(\sqrt{m})$ coordinates. From such a vector $x$, define
$x_{\rightarrow i}$ as the vector having its $j$'th coordinate equal
to the $(j-i)$'th coordinate of $x$ if $j>i$ and otherwise its
$j$'th coordinate is $0$. In words, $x_{\rightarrow i}$ is just $x$
with all coordinates \emph{shifted} by $i$.

For $i \leq n-O(\sqrt{m})$ ($i$ small enough that all the non-zero
coordinates of $x$ stay within the $n$ coordinates), we see that
$TDx$ and $TDx_{\rightarrow i}$ have the exact same
distribution. Furthermore, if $i \geq m + C\sqrt{m}$ for a big
enough constant $C$, $TDx$ and $TDx_{\rightarrow i}$ are
independent. To see this, observe that the only random variables
amongst $t_{-(m-1)},\dots,t_{n-1}$ that are multiplied with a
non-zero coordinate of $x$ in $TDx$ are
$t_{-(m-1)},\dots,t_{O(\sqrt{m})}$. Similarly, only random variables
$d_1,\dots,d_{O(\sqrt{m})}$ amongst $d_1,\dots,d_{n}$ are multiplied
by a non-zero coordinate of $x$. For $x_{\rightarrow i}$, the same
is true for variables $t_{-(m-1)+i},\dots,t_{O(\sqrt{m})+i}$ and
$d_{i+1},\dots,d_{i+O(\sqrt{m})}$. If $i \geq m + C\sqrt{m}$, these
sets of variables are disjoint and hence $TDx$ and
$TDx_{\rightarrow i}$ are independent.

With the observations above in mind, we now define a set of vectors
$X$ as follows
\begin{equation*}
  X := \{\textbf{0}, x, x_{\rightarrow m + C\sqrt{m}}, x_{\rightarrow 2(m+C\sqrt{m})},
  \dots, x_{\rightarrow \lfloor (n-C\sqrt{m})/(m+C\sqrt{m})\rfloor (m + C\sqrt{m})}\}.
\end{equation*}
The $\textbf{0}$-vector clearly maps to the $\textbf{0}$-vector when
using $\frac{1}{\sqrt{m}}TD$ as embedding. Furthermore, by the
arguments above, the embeddings of all the remaining vectors are
independent and have the same distribution. It follows that
\begin{align*}
  \Pr\left[\forall x_{\rightarrow i} \in X :
  \left|\left\|\frac{1}{\sqrt{m}}TDx_{\rightarrow i}\right\|_2^2 -
  1\right|  \leq \eps\right] &\leq \left(1-2^{-\bigO(\eps
                               \sqrt{m})}\right)^{|X|} \\
                             &\leq \exp\left(-|X|2^{-\bigO(\eps
                               \sqrt{m})}\right). \\
\end{align*}
Now since $\textbf{0} \in X$, it follows that to preserve all
pairwise distances amongst vectors in $X$ to within $(1 \pm \eps)$,
we also have to preserve all norms to within $\pm \eps$. This is
true since for all $x$ of unit norm:
\begin{align*}
  \left|\left\|\frac{1}{\sqrt{m}}TDx\right\|_2^2 -
  1\right| &= \left|\left\|\frac{1}{\sqrt{m}}TDx - \frac{1}{\sqrt{m}}TD\textbf{0}\right\|_2^2 -
             \left\|x - \textbf{0} \right\|_2^2\right|.
\end{align*}
This proves the following:
\begin{theorem}\label{thm:mainMany}
  Let $T$ be $m \times n$ Toeplitz and $D$ $n \times n$ diagonal. If
  $t_{-(m-1)},t_{-(m-2)}, \dots,t_{n-1}$ and
  $d_1,\dots,d_n$ are independently distributed Rademacher random
  variables for $i=-(m-1),\dots,n-1$ and $j=1,\dots,n$, then for all
  $0 < \eps < C$, where $C$ is a universal constant, there exists a
  set of $N= \Omega(n/m)$ vectors $X \subset \R^n$ such that
  \begin{equation*}
    \Pr\left[\forall x,y \in X : \left|\left\|\frac{1}{\sqrt{m}}TDx -
          \frac{1}{\sqrt{m}}TDy\right\|_2^2  - \| x-y\|_2^2\right|  \leq
      \eps\|x-y\|_2^2 \right] \leq \exp\left(-N2^{-\bigO(\eps
        \sqrt{m})}\right).
  \end{equation*}
\end{theorem}
It follows from Theorem~\ref{thm:mainMany} that if we want to have constant
probability of successfully embedding \emph{any} set of $N$ vectors,
then either it must be the case that $m = \Omega(n/N)$, or
\begin{align*}
  -N2^{-\bigO(\eps
  \sqrt{m})} &\geq -C_0,
\end{align*}
where $C_0$ is a constant. This in turn implies that
\begin{alignat*}{3}
  \lg N -\bigO(\eps
  \sqrt{m}) &\leq \lg C_0 &&\implies \\
  \sqrt{m} &= \Omega(\eps^{-1} \lg N) &&\implies \\
  m &= \Omega(\eps^{-2}\lg^2 N).
\end{alignat*}
This completes the proof of Theorem~\ref{thm:highlevelmainMany}.

\bibliography{references}

\end{document}